\documentclass[10pt]{amsart}  
\usepackage{hyperref}
\usepackage{latexsym}
\usepackage{amsmath}
\usepackage{amsfonts}
\usepackage{amssymb}

\usepackage{mathrsfs}
\usepackage{array}
\usepackage{caption}
\usepackage{url}

\newtheorem{theorem}{Theorem}[section]
\newtheorem{lemma}[theorem]{Lemma}
\newtheorem{corollary}[theorem]{Corollary}
\newtheorem{proposition}[theorem]{Proposition}

\newtheorem{sublemma}{}[theorem]

\theoremstyle{definition}

\theoremstyle{remark}

\numberwithin{equation}{section}

\usepackage{graphicx}
\DeclareGraphicsExtensions{.pdf,.png,.jpg}
\graphicspath{ {Images/} }

\usepackage{enumerate}

\usepackage[usenames]{color}

\newcommand*\comp[1]{\overline{#1}}

\newcommand{\squid}{F_1}
\newcommand{\lobster}{F_2}
\newcommand{\starfish}{F_3}
\newcommand{\virus}{F_4}
\newcommand{\rocket}{F_5}
\newcommand{\prism}{F_6}

\makeatletter
\def\old@comma{,}
\catcode`\,=13
\def,{%
  \ifmmode%
    \old@comma\discretionary{}{}{}%
  \else%
    \old@comma%
  \fi%
}
\makeatother

\usepackage{cite}

\begin{document}

\title[3-Regular Permutation Graphs]{Characterization and Enumeration of 3-Regular Permutation Graphs}

\author{Aysel Erey}
\address{Department of Mathematics\\
Gebze Technical University\\
Kocaeli, Turkey}
\email{aysel.erey@gtu.edu.tr}

\author{Zachary Gershkoff}
\address{Department of Mathematics\\
Louisiana State University\\
Baton Rouge, Louisiana}
\email{zgersh2@math.lsu.edu}

\author{Amanda Lohss}
\address{Department of Mathematics, Physics, and Statistics\\
Messiah College\\
Mechanicsburg, Pennsylvania}
\email{alohss@messiah.edu}

\author{Ranjan Rohatgi}
\address{Department of Mathematics and Computer Science\\
Saint Mary's College\\
Notre Dame, Indiana}
\email{rrohatgi@saintmarys.edu}

\date{\today}

\begin{abstract}
A permutation graph is a graph that can be derived from a permutation, where the vertices correspond to letters of the permutation, and the edges represent inversions. We provide a construction to show that there are infinitely many connected $r$-regular permutation graphs for $r \geq 3$. We prove that all $3$-regular permutation graphs arise from a similar construction. Finally, we enumerate all $3$-regular permutation graphs on $n$ vertices.
\end{abstract}


\maketitle

\section{Introduction}
\label{Introduction}

The graphs considered here are finite and simple. A graph on $n$ vertices is a {\it permutation graph} if there is a labeling $v_1, v_2, \ldots, v_n$ of the vertices, and a permutation $\pi = [\pi(1), \pi(2), \ldots,  \pi(n)]$, such that $v_i$ and $v_j$ are adjacent in $G$ if and only if $i < j$ and $\pi(i) > \pi(j)$. In this case, the ordered pair $(\pi(i), \pi(j))$ is said to be an {\it inversion} of $\pi$. 
This definition of permutation graphs was given in 1971 by Pneuli et al. \cite{ple}. We note that this is different from the ``generalized prisms'' \cite{Yi} notion of permutation graphs given by Chartrand and Harary \cite{ch}.

Permutation graphs have received a considerable amount of attention in the literature since their introduction (see, for example, \cite{kls, ru, js}). Many algorithmic problems have efficient solutions on permutation graphs. For example, it was shown in \cite{CHP} that the longest path problem (which is NP-complete on general graphs) can be solved in linear time on permutation graphs.


There has been interest in enumerating various types of permutation graphs. For instance, in \cite{kr2}, Koh and Ree gave a recurrence relation for the number of connected permutation graphs. In \cite{AH}, the number of permutation trees is shown to be $2n-2$ for $n\geq 2$. A graph is called {\it $r$-regular} if the degree of the each vertex of the graph is equal to $r$ . It is easy to see that the only connected $2$-regular permutation graphs are $C_3$ and $C_4$ \cite{kr}, as it is well known that permutation graphs cannot have induced cycles of length five or greater. In this direction, we will consider $r$-regular permutation graphs with $r > 2$, we show that the family is infinite for each $r>2$.

\begin{theorem}\label{construction}
For every $r \geq 3$, there are infinitely many connected $r$-regular permutation graphs.
\end{theorem}

We give a complete characterization of $3$-regular permutation graphs in Section~\ref{Characterization}. This will be given in terms of the construction mentioned above.

An interesting corollary of our construction is that almost all $3$-regular permutation graphs are planar. The family of permutation graphs is closed under induced subgraphs (see, for example, \cite{cp}), but a description in terms of minors, as planarity results are normally stated, is not tractable since permutation graphs are not closed under subgraphs.

\begin{corollary}\label{planar}
Every $3$-regular permutation graph except $K_{3,3}$ is planar.
\end{corollary}

Finally, we use the characterization of $3$-regular permutation graphs to enumerate them with a generating function.

\begin{theorem}\label{enumeration}
Let $a(n)$ be the number of $3$-regular permutation graphs on $n$ vertices, and let $A(x)$ be the function

\[
 \frac{1}{2} \Big( \frac{1}{1 - x^2 - x^3} + \frac{1 + x^2 + x^3}{1 - x^4 - x^6} \Big).
\]

\begin{enumerate}[(i)]

\item If $n \in \{4,6\}$, then $a(n) = 1$;
\item If $m = \frac{n-10}{2}$ is a positive integer, then $a(n)$ is given by the coefficient of $x^m$ in the expansion of $A(x)$;

\item Otherwise $a(n) = 0$.
\end{enumerate}
\end{theorem}


Proofs of Theorem~\ref{construction}, Corollary~\ref{planar}, and Theorem~\ref{enumeration} can be found in Sections \ref{Construction}, \ref{Characterization}, and \ref{Enumeration}, respectively.

\section{preliminaries}
\label{preliminaries}

If $G$ is a permutation graph with corresponding permutation $\pi$, we say that $\pi$ is a ${\it realizer}$ of $G$. When discussing a realizer and its graph, we will sometimes refer to a vertex in the graph and an entry in the permutation with the same label. It is well known (for example, in \cite{ple}) that $G$ is a permutation graph if and only if its complement $\comp{G}$ is also a permutation graph.

There are many known characterizations of permutation graphs. Recent characterizations include one by Gervacio et al. \cite{grr} in terms of cohesive vertex-set orders, and one by Limouzy \cite{vl} in terms of Seidel minors. Here we rely on the 1967 characterization by Gallai \cite{tg} in terms of forbidden induced subgraphs (see also \cite{isgci,ar}). All cycle graphs on five or more vertices are forbidden induced subgraphs.
We will refer to these as {\it large holes}.
Table~\ref{Forbid} illustrates all other forbidden induced subgraphs with maximum degree $3$.

\begin{table}[htb] \caption{Forbidden induced subgraphs for permutation graphs with $\Delta \leq 3$} 
	\centering      
	\begin{tabular}{ >{\centering\arraybackslash}m{1.4in}  >{\centering\arraybackslash}m{1.4in}  >{\centering\arraybackslash}m{1.4in} }  
		&&\\
\includegraphics[scale=0.45]{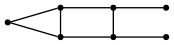} \caption*{$\squid$}& 	\includegraphics[scale=0.4]{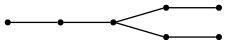} \caption*{$\lobster$}&
\includegraphics[scale=0.45]{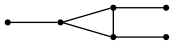} \caption*{$\starfish$}\\

\includegraphics[scale=0.45]{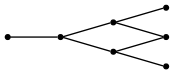} \caption*{$\virus$}&
	\includegraphics[scale=0.45]{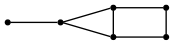} \caption*{$\rocket$}& 	\includegraphics[scale=0.45]{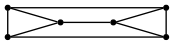} \caption*{$\prism$}\\        
	\end{tabular} \label{Forbid}  
\end{table} 

Throughout this paper, we use $K_i$ and $I_i$ to denote the complete graph on $i$ vertices and the empty graph on $i$ vertices, respectively. We also use $\oplus$ to denote graph disjoint union, and $\Box$ to denote a Cartesian product of graphs. A cycle of length $3$ is referred to as a {\em triangle}, and a cycle of length $4$ is referred to as a {\em square}, regardless of whether or not the cycle is induced.

\section{Infinitely many \texorpdfstring{$r$}{r}-regular permutation graphs for \texorpdfstring{$r\geq 3$}{r > 3}}
\label{Construction}

Let $G$ be a graph of order $n$ with vertices $v_1, v_2, \ldots, v_n$. Given $n$ graphs $H_1, H_2, \dots ,H_n$, we define the {\it composition of $H_1, H_2, \dots ,H_n$ into $G$}, denoted $G[H_1, H_2, \dots ,H_n]$, as the graph which is obtained from $G$ by replacing the vertex $v_i$ with the graph $H_i$. More precisely, the vertex set of $G[H_1, H_2, \dots ,H_n]$ is the disjoint union of the vertex sets of every $H_i$, and $uv$ is an edge of $G[H_1, H_2, \dots ,H_n]$ if and only if either $uv\in E(H_i)$ for some $i$, or there are distinct indices $i$ and $j$ such that $u\in V(H_i)$, $v\in V(H_j) $ and $v_iv_j \in E(G)$. If each graph $H_i$ is a complete graph or empty graph then $G[H_1, H_2, \dots ,H_n]$ is called a {\it blow-up} of $G$, and we say that vertex $v_i$ is {\it blown up} into $H_i$, or {\it replaced} with $H_i$. Notice that $G=G[I_1,_I2,\ldots,I_n]$ and we call this the {\it trivial blow-up} of $G$. We will use {\it blow-up} to mean non-trivial blow-up for the rest of this paper.

\begin{lemma}\label{complemma} \cite[Theorem 3.3]{grr}
Let $G$ be a graph of order $n$ and $H_1, H_2, \, \dots , H_n$ be arbitrary graphs. Then $G^* = G[H_1, H_2, \dots , H_n]$ is a permutation graph if and only if $G$ and each of $H_1, H_2, \, \dots , H_n$ are permutation graphs.
\end{lemma}

    


Using the above lemma, we prove that there are infinitely many connected $r$-regular permutation graphs for every $r \geq 3$.

\begin{proof}[Proof of Theorem~\ref{construction}]
Let $r\geq 3$. For every $n\geq 0$, we construct an $r$-regular permutation graph $G_n$ of order $2nr+r+1$ by taking a blow-up of a path. Let $m = 4n + 2$ and take a path graph $P_m$ with vertices $v_1,v_2,\dots , v_m$ in standard order. Note that $P_m$ is a permutation graph because its maximum degree is $2$ and it does not have an induced subgraph from Table~\ref{Forbid}. Replace the first vertex $v_1$ with $K_2$ and the last vertex $v_m$ with $K_{r-1}$.
For vertices $v_i$ with $i \equiv 2 \pmod 4$, replace them with $I_{r-1}$;
with $i \equiv 3 \pmod 4$, replace $v_i$ with $I_{r-2}$;
with $i \equiv 0 \pmod 4$, replace $v_i$ with $I_1$;
and for $i \equiv 1 \pmod 4$, replace $v_i$ with $I_2$.
The resulting graph $G_n$ is $r$-regular, and since complete graphs and empty graphs are permutation graphs, by Lemma~\ref{complemma}, $G_n$ is a permutation graph. Hence, we obtain an infinite list of $r$-regular permutation graphs

\begin{eqnarray*}
& G_0=P_2[K_2, K_{r-1}] & \\
& G_1=P_6[K_2,I_{r-1},I_{r-2},I_1,I_2,K_{r-1}] & \\
& G_2=P_{10}[K_2,I_{r-1},I_{r-2},I_1,I_2,I_{r-1},I_{r-2},I_1,I_2,K_{r-1}] & \\
& G_3=P_{14}[K_2,I_{r-1},I_{r-2},I_1,I_2,I_{r-1},I_{r-2},I_1,I_2,I_{r-1},I_{r-2},I_1,I_2,K_{r-1}] & \\ & \vdots &
\end{eqnarray*}

and the result follows.
\end{proof}

\section{Characterization of \texorpdfstring{$3$}{3}-regular permutation graphs}
\label{Characterization}

Table~\ref{boxcar_table} shows induced subgraphs we use in our construction of $3$-regular permutation graphs. A {\it boxcar graph} is a graph that can be constructed by the following process.
\begin{enumerate}[(A)]
\item \label{A} Let $S_1 = G_1$. Then go to (\ref{B}).
\item \label{B} Choose $S_2$ to be one of $\{G_2, G_3, G_4\}$. Then go to (\ref{C}).
\item \label{C} Let $G$ be the graph obtained by identifying the rightmost vertex of $S_1$ with the leftmost vertex of $S_2$. If $S_2 = G_4$, go to (\ref{D}); otherwise, set $S_1 = G$ and go to (\ref{B}).
\item \label{D} Stop. The result is the graph $G$.
\end{enumerate}

\begin{table}[htb] \caption{Some induced subgraphs of boxcar graphs} 
	\centering      
	\begin{tabular}{ >{\centering\arraybackslash}m{.8in}  >{\centering\arraybackslash}m{1.1in}  >{\centering\arraybackslash}m{1.3in} >{\centering\arraybackslash}m{1.1in } }  
		&&&\\
		\includegraphics[scale=0.4]{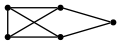} \caption*{$G_1$}&	\includegraphics[scale=0.4]{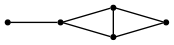} \caption*{$G_2$}& 	\includegraphics[scale=0.4]{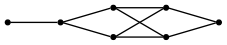} \caption*{$G_3$}&
		\includegraphics[scale=0.4]{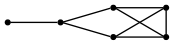} \caption*{$G_4$}\\       
	\end{tabular} \label{boxcar_table}  
\end{table}

\begin{lemma}\label{blowups}
A $3$-regular graph that is a blow-up of a path is isomorphic to $K_4$, $K_{3,3}$, or a boxcar graph.
\end{lemma}
\begin{proof}
Let $G$ be a path graph $P_n$ with vertices $(v_1, v_2, \ldots, v_n)$ in standard order, and consider a blow-up $G^* = G[H_1, H_2, \ldots, H_n]$. There are four possibilities for the graph $H_1$.

Suppose the first vertex $v_1$ is blown up into $K_k$ or $I_k$, with $k \geq 4$. If $v_2$ exists, then the vertices resulting from blowing up $v_2$ will have degree at least $4$. Thus to obtain a $3$-regular graph, $v_1$ must be the only vertex of $G$, and it must be blown up into $K_4$.

Now suppose $H_1 \cong K_3$. Then $v_2$ must be blown up into a graph of order $1$ because the vertices from $K_3$ require one more neighbor to have degree $3$. Since all the vertices have degree $3$, we see that $G$ must be $P_2$, and it blows up into $P_2[K_3, K_1] \cong K_4$.

Suppose $H_1 \cong I_k$, where $k \leq 3$. Since the vertices of $H_1$ require $3$ neighbors, $v_2$ must be blown up into a graph of order $3$. If $H_2 \cong K_3$, then $k = 1$ and we have $K_4$ as in the case above. If $H_2 \cong I_3$, then the vertices of $H_2$ have $k$ neighbors on left and they requre $3-k$ neighbors on the right. In order to not exceed degree $3$, we must have $H_3 \cong I_{k-3}$. Thus we obtain $P_2[I_3, I_3]$, $P_3[I_1, I_3, I_2]$, or $P_3[I_2, I_3, I_1]$, all of which are isomorphic to $K_{3,3}$.

The only remaining cases are when $H_1$
is isomorphic to $K_2$. If $H_1 \cong K_2$, then $H_2$ must have order $2$. If $H_2 \cong K_2$, then we have $G^* = P_2[K_2, K_2] \cong K_4$. If $H_2 \cong I_2$, then $H_3$ must have order $1$, so $H_3 \cong K_1$ and we see that $G^*$ must begin with $G_1$ from Table~\ref{boxcar_table}. Then $H_4 \cong K_1$.
We will use the following sublemma in the remaining cases for the construction of $G^*$.

\begin{sublemma}\label{morebox}
Let $G'$ be $G^*$ restricted to the vertices obtained from blowing up $(v_1, v_2, \ldots, v_k)$ for $k \leq n$. If $G'$ is $3$-regular except for a leaf, then the vertices $(v_{k-1}, v_k, v_{k+1}, v_{k+2})$ or $(v_{k-1}, v_k, v_{k+1}, v_{k+2}, v_{k+3})$ of $G$ must blow up to induce one of $\{G_2, G_3, G_4\}$ in $G^*$.
\end{sublemma}

Assume $v_k$ blows up to be a leaf in $G^*$. Then $H_{k-1} \cong H_{k} \cong K_1$, and $H_{k+1}$ must have order $2$. If $H_{k+1} \cong K_2$, then $H_{k+2}$ must have order $1$, and we have $G_2$ as an induced subgraph on $\cup_{i=k-1}^{k+2} V(H_i)$.
If instead $H_{k+1} \cong I_2$, then we must have either $H_{k+2} \cong I_2$ and $H_{k+3} \cong K_1$, giving us $G_3$, or $H_{k+2} \cong K_2$, giving us $G_4$. Thus \ref{morebox} holds.

By repeated application of \ref{morebox}, we see that $G^*$ consists of a series composition of $G_1$ followed by a sequence of graphs isomorphic to $G_2$ or $G_3$ and terminating in $G_4$.
\end{proof}

The following lemmas will be useful in our characterization of $3$-regular permutation graphs. We say that vertices $v_1$ and $v_2$ are {\it twins} if $N(v_1) - \{v_2\} = N(v_2) - \{v_1\}$, where $N(v_i)$ is the set of vertices that neighbor $v_i$. We do not distinguish between twins that are adjacent and those that are not.


A subsequence $\pi _{i_1},\dots ,\pi_{i_k}$ of a permutation $\pi$ is called {\it consecutive} if $$\pi_{i_{j+1}}=\pi_{i_j}+1 \ \text{or} \ \pi_{i_{j+1}}=\pi_{i_j}-1$$ and it is called {\it contiguous} if $${i_{j+1}}={i_j+1}$$ for all $j=1, \dots , k-1$. 

The following lemma appears to be well known in the field of modular decompositions. We include its proof for completeness.

\begin{lemma} \label{twins}
Every permutation graph $G$ has a realizer $\pi$ where, for every pair of twins $u$ and $v$ in $G$, there is a contiguous, consecutive increasing or decreasing subsequence $s$ of $\pi$ that contains $u$ and $v$. Moreover, $u$ and $v$ are adjacent in $G$ if and only if $s$ is decreasing.
\end{lemma}

\begin{proof}
Let $\pi$ be a realizer of a graph $G$, and define $G_\pi$ to be a graph isomorphic to $G$ with vertex labels corresponding to $\pi$. Let $u$ and $v$ be twins in $G_\pi$ with $u < v$. We will first assume $u$ and $v$ are nonadjacent. If $u$ and $v$ are not part of a contiguous, consecutive increasing subsequence of $\pi$, then we can obtain another realizer $\pi'$ of $G$ by removing $v$ from $\pi$, shifting all of the entries greater than $u$ and less than $v$ up by 1, and inserting $u+1$ to the immediate right of $u$. Clearly $\pi$ and $\pi'$ realize isomorphic graphs, and if $a$ and $b$ are entries of $\pi$ that belong to a contiguous, consecutive increasing or decreasing subsequence of $\pi$, then this transformation does not separate them.


If we assume instead that $u$ and $v$ are adjacent in $G_\pi$, then we apply a similar transformation, ultimately placing $u+1$ to the left of $u$ instead of the right. This results in $u$ and $u+1$ being part of a contiguous, consecutive decreasing subsequence, instead of increasing.
\end{proof}

\begin{lemma} \label{twins2}
If $G^*$ is a graph with maximum degree $d$, and if $G$ is a graph of minimum order such that $G^*$ is a blow-up of $G$, then $G$ has no degree $d$ twins.
\end{lemma}

\begin{proof}
Observe that by our construction of blow-ups given in Lemma~\ref{complemma}, if $G^*$ is a blow-up of $G$, than any realizer of $G$ can be used to obtain a realizer for $G^*$ by blow-up. Let $u$ and $v$ be degree $d$ twins of $G$. By Lemma~\ref{twins}, $G$ has a realizer $\pi$ where $u$ and $v$ are adjacent and consecutive. Let $\{u_1, u_2, \ldots, u_j\}$ and $\{v_1, v_2, \ldots, v_k\}$ be the entries of a realizer $\pi^*$ for $G^*$ obtained by blowing up $u$ and $v$, respectively. Then $u$ must be blown up into $I_j$ and $v$ must be blown up into $I_k$, because if they were blown up into $K_j$ or $K_k$ for $k \geq 2$, then we would have vertices with degree exceeding $d$. Moreover, unless $j = k = 1$, the vertices $u$ and $v$ must be nonadjacent. In the case that $j = k = 1$, $u_1$ and $v_1$ are twins in $G^*$, and they are contiguous and consecutive in $\pi^*$, which means that there is a graph such that $\{u, v\}$ is blown up from a single vertex. If $j$ and $k$ are not both $1$ then, $u_1, u_2, \ldots, u_j, v_1, v_2, \ldots, v_k$ are all twins in $G^*$. Therefore, they are part of a contiguous, consecutive increasing sequence of $\pi^*$, so they can also be blown up from a single vertex. This contradicts the assumption that $G$ has minimum order.
\end{proof}

Recall that a {\it ladder} is a graph $P_2 \Box P_n$, with $2n$ vertices ${u_1, \dots , u_n, v_1, \dots , v_n}$ such that each of $\{u_1, \dots , u_n\}$ and $\{v_1, \dots , v_n\}$ induces a $P_n$, and $u_i$ is adjacent to $v_i$ for each $i=1, \dots , n$. Each edge $u_iv_i$ is called a {\it rung} of the ladder.

\begin{lemma} \label{squares}
A $3$-regular permutation graph $G$ cannot have a ladder with four or more rungs as a subgraph.
\end{lemma}

\begin{proof}
Suppose $G$ has a ladder as a subgraph, and let $u_i$ and $v_i$ be adjacent vertices on the $i$th rung of a maximal ladder for $i$ in $\{1, 2, \ldots, k\}$. We will prove the lemma by considering three propositions.

\begin{enumerate}
\item \label{squares1} A ladder with three or more rungs cannot have an edge between opposite vertices on the same side of the ladder, such as $v_1$ and $v_k$.
\item \label{squares2} A ladder with four or more rungs cannot have an edge between opposite vertices on the different sides of the ladder, such as $v_1$ and $u_k$.
\item \label{squares3} There cannot be a ladder with three or more rungs without an edge between the first and last rung of the ladder.
\end{enumerate}


To prove proposition (\ref{squares1}), suppose that $k = 3$. Let $v_1$ and $v_k$ be adjacent, and suppose first that $u_1$ and $u_k$ are not. Then we have a large hole using vertices $\{v_1, u_1, u_2, u_3, v_3\}$. However, if $u_1$ and $u_k$ are also adjacent, then we have $\prism$.
Next suppose $k = 4$. If $(v_1, v_k)$ is an edge and $(u_1, u_k)$ is not, then $\{v_1, u_1, u_3, u_4, v_4\}$ is a large hole. If $(u_1, u_k)$ is also an edge, then the graph is isomorphic to a cube ($C_4 \Box K_2$), which has $C_6$ as an induced subgraph by deleting a pair of opposite vertices.
Finally, suppose $k \geq 5$. Then $\{v_1, v_2, \ldots, v_k\}$ is a large hole.


Similarly, for proposition (\ref{squares2}), if $v_1$ and $u_k$ are adjacent, we have a large hole using $\{v_1, v_2, u_2, u_3, \ldots, u_k\}$.


Finally, for proposition (\ref{squares3}), suppose $v_1$ and $u_1$ have a common neighbor $v$. Then $v$ cannot have $v_k$ or $u_k$ as neighbors, or else we have a large hole. So $v$ has another neighbor $v'$, but this gives us $\rocket$ using $\{v', v, v_1, v_2, u_1, u_2\}$. Suppose instead that the third neighbors of $v_1$ and $u_1$ are $v$ and $u$, respectively, with $v \neq u$. Then we have $\virus$ using $\{v, v_1, v_2, u, u_1, u_2, u_3\}$.
\end{proof}

We now prove that the graphs from Lemma~\ref{blowups} are the only $3$-regular permutation graphs.

\begin{theorem}
Every connected $3$-regular permutation graph is the blow-up of a path.
\end{theorem}

\begin{proof}
Suppose $G^*$ is a $3$-regular permutation graph that is not a blow-up of a path. Let $G$ be a permutation graph of minimum order such that $G^*$ is a blow-up of $G$. Then $G$ is either a cycle or $G$ has a degree $3$ vertex.

If $G$ is a cycle, then $G$ must be $C_3$ or $C_4$, because larger cycles are forbidden as induced subgraphs. In $C_3$, since all the vertices are adjacent to each other and they all have degree $2$, only one vertex can be blown up or else we would have a vertex with degree exceeding $3$. Moreover, the vertex must be blown up into $K_2$ in order for every vertex to have degree $3$. The resulting graph is $K_4$. In $C_4$, since every vertex has degree $2$, at most one of the neighbors of every vertex can be blown up. If a vertex $v$ is blown up into $K_2$, then everything in the resulting graph will have degree $3$ expect for the vertex that was opposite of $v$, and it is impossible to use further blow-ups to obtain a $3$-regular graph. Thus, the only possibility that gives a $3$-regular graph is blowing up each of two adjacent vertices into $I_2$. This gives a graph isomorphic to $K_{3,3}$.

Suppose instead that $G$ has a vertex $v$ of degree $3$. Note that $G$ has maximum degree $3$. We proceed by analyzing the possible induced subgraphs $H$ containing $v$ and its neighborhood. In some cases, $H$ has degree-$3$ twins, which by Lemma~\ref{twins2} contradicts the assumption that $G$ is a minimal-order graph that blows up into $G^*$. In some of the remaining cases, $H$ can be blown up (perhaps trivially) into an induced subgraph of a boxcar graph. If $H$ cannot be blown up this way, then we give one of two reasons why. Either $H$ is not an induced subgraph of a $3$-regular permutation graph and cannot be blown up without creating a vertex of degree $4$ or greater, or $H$ has a large hole or forbidden induced subgraph from Table~\ref{Forbid}.

Let the neighbors of $v$, $N(v)$, be $\{v_1, v_2, v_3 \}$. We will consider the following cases based upon the possible subgraphs induced by $N(v)$: 
\begin{enumerate}
    \item $P_3$, 
    \item $K_3$, 
    \item $I_3$, and
    \item $K_2\oplus I_1$.
\end{enumerate}

In cases (1) and (2), there are twin vertices of degree $3$, contradicting Lemma~\ref{twins2}.

In case (3), $N(v)$ induces $I_3$, that is, none of the vertices in $N(v)$ are adjacent.
Observe that if $v$ is adjacent to a leaf in $G$, then $v$ must be blown up into $I_3$ in order to obtain a $3$-regular graph. This implies that all the neighbors of $v$ in $G$ must be leaves or else we would have a vertex of degree exceeding $3$ upon blowing up $v$ to $I_3$. The resulting graph of this blow-up is $K_{3,3}$. Thus we may assume that all vertices adjacent to a degree $3$ vertex have degree at least $2$.

We will proceed by considering the number of squares that contain $v$ as a vertex. If $v$ is not involved in any squares, then the subgraph induced by $N(v)$ and its neighbors has either $\lobster$ or a large hole as an induced subgraph. Suppose instead that $\{v, v_2, v_3\}$ are used in a square. Let $v_4$ be the remaining vertex of the square. In the case that $v_4 = v_1$, then we have an induced subgraph of $G_2$ from Table~\ref{boxcar_table}; in particular, it is the blow-up of a path. In the remaining case where $v_4$ is a distinct vertex, if each one of $\{v_2, v_3, v_4\}$ has degree $2$, then $v_4$ can be blown up into $K_2$ to realize $G_4$ from Table~\ref{boxcar_table}. Similarly, if $v_2$ and $v_3$ have degree $2$ and $v_4$ has degree $3$, then $v_4$ can be blown up into $I_2$ to realize $G_3$. If, however, only one of $\{v_2, v_3\}$ has degree $3$, or they both have degree $3$ and $v_4$ has degree $2$, then we cannot perform any more blow-ups without creating a vertex with degree exceeding $3$.
If all of $\{v_2, v_3, v_4\}$ have degree $3$, then depending the configuration of the remaining edges, we either have $\virus$, $\rocket$, or a large hole as an induced subgraph. 

Suppose $v$ is used in two squares. One possibility is for two neighbors of $v$ to be involved in both squares; say $\{v, v_2, v_3\}$ are involved in two distinct squares. This implies that $v_2$ and $v_3$ are degree $3$ twins. Another possibility is for the two squares to share a single edge, creating a ladder subgraph with at least three rungs. Observe that if the largest ladder subgraph using $v$ has three rungs, and there is an edge between two opposite vertices in a cycle around the ladder, then $v$ is involved in at least three squares. The remaining possibilities for a ladder on three or more rungs using $v$ contradict cases (\ref{squares1}), (\ref{squares2}), and (\ref{squares3}) from Lemma~\ref{squares}.

The final possibilities when $N(v)$ induces $I_3$ are for $v$ to be involved in three or more squares. If $v$ is in exactly $3$ squares, either there is an induced $6$-cycle around $v$, or the vertices at distance $2$ or less from $v$ induce $G_3$ as a subgraph. If $v$ is used in more than three squares, then $G \cong K_{3,3}$.

In case (4), $N(v)$ induces $K_2 \oplus I_1$. Suppose that $\{v, v_2, v_3\}$ forms a triangle. If one of $\{v_2, v_3\}$ has degree $2$, then the graph cannot be blown up to be $3$-regular. If they both have degree $3$ and are not in a square with $v_1$, then either they have a common neighbor other than $v$, giving us $G_2$ from Table~\ref{boxcar_table}, or they have different neighbors, giving us $\starfish$. Suppose that $\{v, v_2, v_3\}$ is a triangle and $\{v, v_1, v_2, v_4\}$ is a square for some new vertex $v_4$. If $v_3$ has degree $2$, then $G$ cannot be blown up into a $3$-regular graph. If $v_3$ is adjacent to $v_4$, then this is isomorphic to $G_1$. If $v_3$ is adjacent to a new vertex $v_5$, then we have $\rocket$ as an induced subgraph.

Finally, let $\{v, v_2, v_3\}$ be a triangle, and suppose there are squares $\{v, v_1, v_2, v_4\}$ and $\{v, v_1, v_3, v_5\}$. If $v_4 = v_5$, then $v_4$ and $v$ are twins; a contradiction. Suppose $v_4 \neq v_5$. If $v_4$ and $v_5$ are nonadjacent, then $\{v_1, v_2, v_3, v_4, v_5\}$ is a large hole, and if they are adjacent, then our induced subgraph is isomorphic to $\prism$.
\end{proof}

This theorem and Lemma~\ref{blowups} immediately imply the following corollary.

\begin{corollary}\label{characterization}
Every $3$-regular permutation graph is isomorphic to $K_4$, $K_{3,3}$, or a boxcar graph.
\end{corollary}

Note that this also implies Corollary~\ref{planar}, since every boxcar graph has a planar embedding.

\begin{corollary}
Every $3$-regular permutation graph has a Hamiltonian path.
\end{corollary}

\begin{proof}
Clearly $K_4$ and $K_{3,3}$ are Hamiltonian. Observe that every graph in $\{G_1, G_2, G_3, G_4\}$ from Table~\ref{boxcar_table} also has a Hamiltonian path. By merging the degree $1$ vertices to obtain a boxcar graph, we find that Hamiltonian paths of each of the graphs $\{G_1, G_2, G_3, G_4\}$ connected in sequence give a Hamiltonian path for the boxcar graph.
\end{proof}

\section{Enumeration of connected \texorpdfstring{$3$}{3}-regular permutation graphs}
\label{Enumeration}
Using the characterization given in Corollary~\ref{characterization} and a generating function to enumerate $3$-regular permutation graphs, we prove Theorem~\ref{enumeration}. In what follows, we will use {\it sequences for $m$} (where $m$ is a positive integer) to mean equivalence classes of compositions of $m$ into parts of size $2$ and $3$ where a composition and its reverse are considered to be the same. For example, the integer $11$ has five equivalence classes and their representatives are $\gamma _1=(2,3,3,3)$, $\gamma _2=(3,2,3,3)$, $\gamma _3=(3,2,2,2,2)$, $\gamma _4=(2,3,2,2,2)$, $\gamma _5=(2,2,3,2,2)$. The sequence $\gamma _1'=(3,3,3,2)$ belongs to the equivalence class of $\gamma _1$, as $\gamma _1$ and $\gamma _1'$ are reverses of each other.

Before we present the generating function, we first show that sequences for $m$ are in $1-1$ correspondence with boxcar graphs.

\begin{proposition}\label{boxseq}
Let $m = \frac{n-10}{2}$. For even $n$, the number of boxcar graphs on $n$ vertices is equal to the number of sequences for $m$.
\end{proposition}

\begin{proof}
We can think of a boxcar graph as starting with $G_1$, followed by a sequence of $G_2$ and $G_3$ subgraphs, and ending with $G_4$ (see Table \ref{boxcar_table}). Notice that $G_1$ and $G_4$ contribute 5 vertices each to the boxcar graph while each copy of $G_2$ and $G_3$ contribute 4 and 6, respectively, after identification of vertices. Thus the isomorphism class of the graph is determined by the sequence of $G_2$ and $G_3$ subgraphs in the middle. To count boxcar graphs, we wish to find the integer compositions of $n-10$ into parts of size $4$ and $6$, which is equivalent to the integer compositions of $\frac{n-10}{2}$ into parts of size $2$ and $3$. Moreover, since the start and end of the sequence of subgraphs in a boxcar graph are not distinguished, we count a composition and its reverse as being the same.
\end{proof}

Techniques used for some omitted computations below may be found in \cite{wilf}.

\begin{proposition}\label{enumerateseq}
The generating function for sequences for $m$ is

\[
A(x) = \frac{1}{2} \Big( \frac{1}{1 - x^2 - x^3} + \frac{1 + x^2 + x^3}{1 - x^4 - x^6} \Big).
\]
\end{proposition}

\begin{proof}
Since we want to count a composition and its reverse as the same, we may count all compositions that are symmetric, and half of all the non-symmetric ones. This is equivalent to counting half of all compositions, and adding again half of all the symmetric ones.

Let $t_n$ denote the number of compositions of $n$ into parts of size $2$ and $3$. Note that these compositions of $n$ are in bijection with the compositions of $n-2$ and $n-3$ into parts of size $2$ and $3$, because we can obtain such a composition of $n-2$ or $n-3$ by taking such a composition of $n$ and removing the first part, and this operation has a well-defined inverse. Thus $t_n = t_{n-2} + t_{n-3}$. Deriving a generating function $T(x)$ from this recursion relation with $t_0 = t_2 = 1$ and $t_1 = 0$ gives $T(x) = \frac{1}{1 - x^2 - x^3}$.

Observe that a symmetric composition may have an even number of parts, or it may have an odd number of parts with a $2$ or $3$ in the middle. Moreover, each part not in the middle must have an identical part on the opposite end of the composition. Therefore the symmetric compositions of $n$ into parts of size $2$ and $3$ are in bijection with the compositions of $n-k$ into parts of size $4$ and $6$ for all $k$ in $\{0,2,3\}$.

If we let $u_n$ be the number of compositions of $n$ into parts of size $4$ and $6$, we obtain the recurrence relation $u_n = u_{n-4} + u_{n-6}$. Taking $u_0 = u_4 = 1$ and $u_1 = u_2 = u_3 = u_5 = 0$, we can derive the generating function $U(x) = \frac{1}{1 - x^4 - x^6}$. Then the generating function for the number of symmetric compositions is $V(x) = U(x) + x^2 U(x) + x^3 U(x) = \frac{1 + x^2 + x^3}{1 - x^4 - x^6}$. Thus we have $A(x) = \frac{1}{2} (T(x) + V(x))$.
\end{proof}

The above two propositions are used to complete the proof of Theorem~{\ref{enumeration}}.

\begin{proof}[Proof of Theorem \ref{enumeration}]
Corollary~\ref{characterization} tells us that the only $3$-regular permutation graphs that are not boxcar graphs are $K_4$ and $K_{3,3}$, which have $4$ and $6$ vertices, respectively. By Proposition~\ref{boxseq}, the problem of counting boxcar graphs can be reduced to the problem of counting sequences for $m$, which is done in Proposition~\ref{enumerateseq}.
\end{proof}


\section{Conclusion}
\label{Conclusion}
We have proven that there are infinitely many $r$-regular permutation graphs for $r\geq 3$ and given a complete characterization of $3$-regular permutation graphs in terms of blow-ups of paths. While it is perhaps surprising that all $3$-regular permutation graphs are blow-ups paths, this is not the case for all $r$-regular graphs in general. In particular, Figure~\ref{Counterexample} is a counterexample with $r=4$.

\begin{figure}[h]
\includegraphics[scale=0.45]{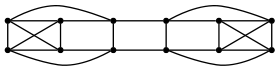}
\caption{A $4$-regular permutation graph that is not a blow-up of a path ($\pi=[5,4,7,2,1,10,3,12,11,6,9,8]$).}\label{Counterexample}
\end{figure}

The graph from Figure~\ref{Counterexample} can be constructed by blowing up a $4$-runged ladder. More specifically, if $G$ is the $4$-runged ladder whose vertices are labeled as they appear in a Hamiltonian path starting and ending on a degree $2$ vertex, then the graph from Figure~\ref{Counterexample} is $G[K_2,K_1,K_1,K_2,K_2,K_1,K_1,K_2]$. Note that $G$ is a permutation graph with realizer $[3,5,1,7,2,8,4,6]$. This observation, along with the lemma below, indicates that the permutation graph from Figure~\ref{Counterexample} is not the blow-up of a path.

\begin{lemma}\label{counter}
For each graph $G$, there is unique graph $G'$ of minimal order such that $G$ is a blow-up of $G'$.
\end{lemma}
\begin{proof}
Let $P = (p_1, p_2, \ldots, p_m)$ be the partition of $V(G)$ such that two vertices are in the same part if and only if they are twins. We construct an $m$-vertex graph $G'$, where distinct vertices $v_i, v_j$ of $V(G')$ are adjacent if and only if the members of $p_i$ and $p_j$ are adjacent in $G$. Then $G$ is a blow-up of $G'$, obtained by replacing each vertex $v_i$ with the vertices of $p_i$. We know that $G'$ is minimal because if $H$ is a graph such that $G$ is a blow-up of $H$, and $u_1$ and $u_2$ are vertices of $G$ that arise from the same vertex of $H$, then $u_1$ and $u_2$ must be twins. Moreover, $G'$ is unique because $P$ is unique.
\end{proof}

By taking complements of the graphs listed in from Corollary~\ref{characterization} and applying Lemma~\ref{counter}, we find other counterexamples for $r$-regularity for certain even values of $r$. Counterexamples to show that not every $r$-regular permutation graph is a blow-up of a path for odd values of $r > 4$ are not known.

\section{Acknowledgements}
\label{Acknowledgements}
We gratefully acknowledge the financial support from the following grants that made this research possible: NSF-DMS Grants 1604458, 1604773, 1604697 and 1603823, ``Rocky Mountain-Great Plains Graduate Research Workshops in Combinatorics'' (all authors), and National Security Agency Grant H98230-16-1-0018, ``The 2016 Rocky Mountain-Great Plains Graduate Research Workshop in Combinatorics'' (Amanda Lohss). Generous support was also given by the Institute for Mathematics and its Applications.

We thank the organizers of the Graduate Research Workshop in Combinatorics (GRWC) 2016 and the other participants who give insight into the problem. We also thank Irfan Alam, Mark Ellingham and Karl Mahlburg for consultation on how to enumerate this graph class.

\bibliography{references}
\bibliographystyle{abbrv}
\end{document}